\documentclass[10pt]{article}

\usepackage{amsfonts,amssymb}
\input{xy}
\usepackage[all]{xy}
\xyoption{all}
\xyoption{poly}

\newtheorem{thm}{Theorem}[section]
\newtheorem{cor}[thm]{Corollary}
\newtheorem{lem}[thm]{Lemma}
\newtheorem{pro}[thm]{Proposition}
\newtheorem{defi}[thm]{Definition}%[section]
\newtheorem{rem}[thm]{Remark}%[section]
\newtheorem{exa}[thm]{Example}

\newenvironment{proof}{\noindent \textbf{{Proof.}} \sf}

\def\qed{\hfill $\diamond$ \bigskip}

\def\C{{\mathcal C}}
\def\D{{\mathcal D}}

\def\F{{\mathcal F}}

\def\L{{\mathcal L}}
\def\U{{\mathcal U}}

\def\lim{\mathop{\rm lim}\nolimits}

\def\Hom{\mathop{\sf Hom}\nolimits}

\def\HH{\mbox{{\rm HH}}}
\def\H{\mbox{{\rm H}}}

\def\id{\mathsf{I}}

\begin{document}

\sf

\title{Fundamental group of Schurian categories and the Hurewicz isomorphism}
\author{Claude Cibils, Mar\'\i a Julia Redondo and Andrea Solotar
\thanks{\footnotesize This work has been supported by the projects  UBACYTX212, PIP-CONICET 112- 200801-00487, PICT-2007-02182 and MATHAMSUD-NOCOMALRET.
The second and third authors are  research members of
CONICET (Argentina). The third author is a Regular Associate of ICTP Associate Scheme.}}

\date{}

\maketitle

\begin{abstract}

Let $k$ be a field. We attach a CW-complex to any Schurian $k$-category and we prove that the fundamental group of this CW-complex is isomorphic to the intrinsic fundamental group of the $k$-category. This extends previous results by J.C. Bustamante in \cite{bu1}. We also prove that the Hurewicz morphism from the vector space of abelian characters of the fundamental group to the first Hochschild-Mitchell cohomology vector space of the category is an isomorphism.

\end{abstract}


\section{\sf Introduction}

In this paper we consider Schurian categories, that is, small linear categories over a field $k$ such that each vector space of morphisms is either of dimension one or zero.

Recall that there is no homotopy theory available for a $k$-algebra or, more generally, for a $k$-linear category.  More precisely there is neither homotopy equivalence nor a definition of loops  as in algebraic topology taking into account the $k$-linear structure.  As an alternative we consider an intrinsic fundamental group \textit{\`{a} la} Grothendieck, that we have defined in \cite{CRS} and \cite{CRS2} using connected gradings.  In \cite{CRS2} we have computed this group for matrix algebras $M_p(k)$ where $p$ is a prime and $k$ is an algebraically closed field  of characteristic zero, obtaining that $\pi_1(M_p(k))=F_{p-1}\times C_p$ where $F_{p-1}$ is the free group with $p-1$ generators and $C_p$ is the cyclic group of order $p$, using classifications of gradings  provided in \cite{baza,bodas07,kabodas}.

The intrinsic fundamental group is defined in terms of Galois coverings provided by connected gradings.  It is  the automorphism group of the fibre functor over a fixed object.  In case a universal covering $\U$ exists for a $k$-linear category $\C$, its fundamental group $\pi_1(\C)$ is isomorphic to the automorphism group of the covering $\U$.

It is intrinsic in the sense that it does not depend on the presentation of the $k$-category by generators and relations. If a universal
covering exists, then we obtain that the fundamental groups constructed by R. Mart\'inez-Villa and J.A. de la Pe\~{n}a (see \cite{MP}, and \cite{asde,boga,ga}) depending on the choice of a presentation of the category by a quiver and relations are in fact quotients of the
intrinsic $\pi_1$ that we introduce. Note that those groups can vary according to different presentations of the same
$k$-category (see for instance \cite{asde,buca,le1}).

The definition of  $\pi_1(\C)$ is inspired in the topological case considered for instance in R.~Douady and A.~Douady's book \cite{dodo}.
They are closely related to the way in which the fundamental group is viewed in algebraic geometry
after A.~Grothendieck and C.~Chevalley.

Note that the existence of a universal covering for a $k$-linear category is equivalent to the existence of a universal grading, namely a connected grading such that any other connected grading is a quotient of it.

In this paper we will prove that a Schurian category $\C$ admits a universal covering.  In fact a universal grading is obtained through the topological fundamental group of a CW-complex $CW(\C)$ that we attach to $\C$.  We infer that $\pi_1(\C)=\pi_1(CW(\C))$.  The CW-complex we define is very close to a simplicial complex.  It is a simplicial complex when $\C$ is such that if $_y \C_x \not= 0$ then  $_x \C_y = 0$ for $x \not=y$ (where  $_y \C_x$ is the vector space of morphisms from $x$ to $y$).

J.C.~Bustamante considers in \cite{bu1} $k$-categories with a finite number of objects subject to the above conditions which he calls
"Schurian almost triangular" in order to prove a similar result through the fundamental group of a presentation as defined in \cite{boga,gr,MP}.
He uses the simplicial complex from \cite{bo,bg,madelp} whose $2$-skeleton coincides with ours in the Schurian almost triangular context.
We do not require that the category has a finite number of objects, neither that it admits an admissible presentation.
Moreover we provide an example of a Schurian category which has no admissible presentation and we compute its fundamental group.
Note also that we generalize the result by M.~Bardzell and E.~Marcos in \cite{ba}, where they prove that the fundamental group of a Schurian
basic algebra does not depend on the admissible presentation.

We thank G.~Minian for interesting discussions, and in particular for pointing out that cellular approximation enables to provide homotopy
arguments from algebraic topology using the 1-skeleton. In \cite{bu1}, J.C.~Bustamente uses the edge path group instead, which requires a
finite number of objects. In \cite{bu2} a CW-complex attached to a presentation of a finite number of objects category is considered
in order to compute the fundamental group of the presentation.

In case of a complete Schurian category $\C$, that is all the vector spaces of morphisms are one dimensional and composition
of non-zero morphisms is non-zero, the CW-complex attached to $\C$ enables to retrieve that $\pi_1(\C)=1$, see \cite[Corollary 4.6]{CRS2}.

Finally we consider the Hurewicz morphism (see \cite{asde,cire,dlpsa}) for a Schurian category $\C$.  We show that this morphism from the
vector space of abelian characters of $\pi_1(\C)$ to the first Hochschild-Mitchell cohomology vector space of $\C$ is an isomorphism.

Even though the best understood class of coverings is that of Galois coverings, general non-Galois coverings have also been considered.  For instance, in \cite{doha, japre}, almost Galois coverings and balanced coverings, respectively, have been defined. The approach is different since the focus is to get results in the representation theory of algebras, but they also use gradings, and some notions and results may have a connection with some parts of this paper.

The authors want to thank the referee for the comments and suggestions he/she made, which contributed to improvements in the presentation of the results.

\section{\sf Fundamental group}

Recall that, given a field $k$, a $k$-category is a small category $\C$ such that each morphism set ${}_y\C_x$ from an object
$x\in\C_0$ to an object $y\in\C_0$ is a $k$-vector space, the composition of
morphisms is $k$-bilinear and the $k$-multiples of the identity at each object are central
in its endomorphism ring.

A \textbf{grading} $X$ of a $k$-category $\C$ by a group $\Gamma\!_{X}$ is a direct sum decomposition
\[{}_{y} \C_x =\bigoplus_{s\in \Gamma\!_X} X^s{}_{y} \C_x\]
for each $x,y\in \C_0$, where $X^s{}_{y}\C_x$ is called the  \textbf{homogeneous component} of degree $s$ from $x$ to $y$,
such that for $s,t  \in \Gamma\!_X$
\[X^t{}_{z} \C_{y}\    X^s{}_{y} \C_x \subset  X^{ts}{}_{z} \C_x.\]
In case $f \in X^s{}_{y}\C_x$ and $f\neq 0$ we write $\deg_X f = s$ and we say that $f$ is \textbf{homogeneous} of degree $s$.

In order to define a connected grading, we consider \textbf{virtual morphisms}.  More precisely, each non-zero morphism $f$ from its source $s(f)=x$ to its target $t(f)=y$ gives rise to a  virtual  morphism $(f,-1)$ from $y$ to $x$, and we put $s(f,-1)=y$ and $t(f,-1)=x$.  We consider neither zero virtual morphisms nor composition of virtual morphisms.
A  non-zero usual morphism $f$ is identified with the virtual  morphism $(f,1)$ with the same source and target objects as $f$.

A \textbf{walk} $w$ in $\C$ is a sequence of virtual morphisms
$$(f_n, \epsilon_n), \dots, (f_1, \epsilon_1)$$
where $\epsilon_i \in \{ +1, -1 \}$, such that the target of $(f_i, \epsilon_i)$ is the source of $(f_{i+1}, \epsilon_{i+1})$. We put
$s(w) = s(f_1, \epsilon_1)$ and $t(w)=t(f_n, \epsilon_n)$.

The category $\C$ is \textbf{ connected} if for any pair of objects $(x,y)$ there exists a walk $w$ from $x$ to $y$.

A \textbf{homogeneous virtual morphism} is a virtual morphism $(f, \epsilon)$ with $f$ homogeneous. We put $\deg_X (f, 1)= \deg_X (f)$ and $\deg_X (f, -1)= \deg_X (f)^{-1}$. A \textbf{homogeneous walk} $w$ is a walk made of homogeneous virtual morphisms, and its degree is the ordered product
of the degrees of the virtual morphisms.

By definition the grading $X$ is \textbf{connected} if for any pair of objects $(x,y)$ and any group element $s\in\Gamma\!_X$ there exists a homogeneous walk $w$ from $x$ to $y$ such that $\deg_Xw=s$.  Hence if a connected grading exists the category is necessarily connected. In case the category $\C$ is  already connected,  a grading is connected if for a fixed pair of objects $(x_0,y_0)$ there exists a homogeneous walk from $x_0$ to $y_0$ of degree $s$ for any $s \in \Gamma\!_X$, see \cite{CRS}.

In general a linear category does not admit a universal covering.  However, in case a universal covering $\U$ exists, according to the theory developed in \cite{CRS,CRS2}, we have that the intrinsic fundamental group $\pi_1(\C)$ is isomorphic to the automorphism group of the universal covering.  In this paper we will not provide the general definition of the fundamental group since we will only consider $k$-categories with a universal covering.

\section{\sf CW-complex}

Let $\C$ be a connected \textbf{Schurian} $k$-category, that is a small linear category over a field $k$ such that each vector space of morphisms is either of dimension one or zero.  We choose a non-zero morphism ${}_ye_x$ in each one-dimensional space of morphisms ${}_y\C_x$, where  ${}_xe_x=  {}_x\id_x$ is the unit element of the endomorphism algebra of $x$.

Observe that  ${}_xe_y \ {}_ye_x  \not= 0$ is equivalent to ${}_ye_x \ {}_xe_y  \not= 0$, since if ${}_xe_y \ {}_ye_x=\lambda\left({}_x\id_x\right)$ with $\lambda\in k^{\ast}$, then  ${}_ye_x \ {}_xe_y \not= 0$ since otherwise ${}_ye_x \ {}_xe_y \ {}_ye_x$ is simultaneously zero and a non-zero multiple of ${}_ye_x$.

\begin{defi}
The \textbf{associated CW-complex} $CW(\C)$ is defined as follows
\begin{itemize}
\item The $0$-cells are given by the set of objects $\C_0$.
\item Each morphism  ${}_ye_x$ with $x \not= y$ gives rise to a $1$-cell still denoted ${}_ye_x$ attached to $x$ and $y$.
\item
If $x,y$ and $z$ are pairwise distinct objects such that ${}_y\C_x$, ${}_z\C_y$ and ${}_z\C_x$ are non-zero, and ${}_ze_y \ {}_ye_x \not= 0$, we add a $2$-cell attached to the $1$-cells ${}_ye_x$, ${}_ze_y$ and ${}_ze_x$.
\item
If $x$ and $y$ are distinct objects such that ${}_y\C_x$ and ${}_x\C_y$ are non-zero, and ${}_xe_y \ {}_ye_x \not= 0$ (equivalently ${}_ye_x \ {}_xe_y \not= 0$, as mentioned above), we add exactly one $2$-cell attached to the $1$-cells ${}_ye_x$ and ${}_xe_y$.
\end{itemize}
\end{defi}

\begin{rem}
Note that in case $x$ and $y$ are distinct objects such that ${}_y\C_x \not= 0 \not= {}_x\C_y$, two $1$-cells are attached to $x$ and $y$.

Observe that in case $x,y$ and $z$ are distinct objects such that ${}_ze_y \ {}_ye_x = 0$, there is no $2$-cell attached, even in case ${}_z\C_x \not = 0$.
\end{rem}

The associated CW-complex we have just defined has no $n$-cells for $n\geq 3$, it coincides with its $2$-skeleton.
We do not need to go further since the fundamental group of any CW-complex coincides with the fundamental group of its $2$-skeleton, see for instance \cite[ Chapter 2]{H}.

%\normalsize

\begin{exa} (see \cite[Corollary 4.6]{CRS2})
Let $\D^n$ be a \textbf{complete Schurian category} with $n$-objects $1,\dots, n$: for each pair of objects
$(x,y)$, the morphism space $\D^n_x$ is one dimensional with a basis element ${}_ye_x$, where ${}_xe_x={}_x\id_x$. Composition is defined by ${}_ze_y \ {}_ye_x= {}_ze_x$ for any triple of objects. Note that the direct sum algebra of morphisms for $\D^n$ is the matrix algebra $M_n(k)$. \\
We assert that $CW(\D^n)$ is contractible, that is,  it has the homotopy type of a point. Note that $CW(\D^2)$ is a disk. For $n\geq 3$ consider the CW-subcomplex $\L^n$ consisting of all $0$-cells of $\D^n$ and a chosen $1$-cell attached to $i$ and $i+1$ for $i=1,\dots,n-1$ (there are no $2$-cells in $\L^n$). This CW-subcomplex is closed and contractible. Consequently the quotient $CW(\D^n)/\L^n$ has the same homotopy type than $CW(\D^n)$, see for instance \cite[p.11]{H}.  Moreover $CW(\D^n)/\L^n$ has only one $0$-cell. We assert that each $1$-cell not in $\L^n$ is the border of at least one disk in $CW(\D^n)/\L^n$. Indeed, in case of the $1$-cell not in $\L^n$ between $j$ and $j+1$, for $j=1,\dots,n-1$, the $2$-cell attached to the two $1$-cells between $j$ and $j+1$ becomes the required disk in the quotient.  In case the $1$-cell is between $j$ and $j+k$, for $j=1,\dots,n-2$ with $k=2,\dots,n-j$, the $2$-cells given by the triples $(j,j+1,j+2),(j,j+2,j+3), \dots, (j,j+k-1,j+k)$ provide a disk in the quotient having the original $1$-cell as border. Finally there are two $1$-cells attached to $n$ and $1$, both are not in $\L^n$ and can be identified since a $2$-cell is attached to them; they are the border of the disk obtained with the $2$-cells $(1,2,3),(1,3,4),\dots,(1,n-1,n)$.
\end{exa}

Let $w=(f_n, \epsilon_n), \cdots, (f_1, \epsilon_1)$ be a walk in $\C$ from $x$ to $y$.  The \textbf{inverse walk} $w^{-1}$ is the walk $ (f_1, -\epsilon_1), \cdots,  (f_n, -\epsilon_n)$ from $y$ to $x$.  Note that in case $w$ is a homogeneous walk for a grading $X$, then
$$\deg_X w^{-1} = (\deg_X w)^{-1}.$$

Let $\C$ be a connected  $k$-category and let $X$ be a grading of $\C$. Let $c_0$ be an object of $\C$. A set of \textbf{connector walks}
is a set of walks $u=\{{}_xu_{c_0}\}_{x\in\C_0}$ where ${}_xu_{c_0}$ goes from $c_0$ to $x$, such that $\deg_X {}_xu_{c_0} = 1$ and ${}_{c_0}u_{c_0}={}_{c_0}\id_{c_0}$. If the grading is connected a set of connector walks exist.

Let $\C$ be a $k$-category, $x$ an object in $\C$ and let $w=(f_n, \epsilon_n), \cdots, (f_1, \epsilon_1)$ be a closed walk in $\C$ from $x$ to $x$.
In $CW(\C)$ there is a loop counterpart to $w$  that we still denote $w$ and that we call the \textbf{loop described by $w$} which is defined as
follows. This loop is obtained as the continuous map from  $[0,1]$ subdivided in $n$ intervals $I_i =[\frac{i-1}{n},\frac{i}{n} ]$,
where $I_i$ corresponds to the $1$-cell defined by the non-zero morphism  $f_i$ corresponding to the virtual one $(f_i, \epsilon_i)$ and where
$w(\frac{i-1}{n}) = s(f_i, \epsilon_i)$ and $w(\frac{i}{n})=t(f_i, \epsilon_i)$.

\begin{pro} \label{zeta}
Let $\C$ be a connected Schurian $k$-category, let $X$ be a connected grading of $\C$ and let $u$ be a set of connector walks for $X$ for an object $c_0$. There exists a connected grading $Z_{X,u}$ of $\C$ by the group $\pi_1(CW(\C),c_0)$, where $c_0$ is considered as a base point of the CW-complex.
\end{pro}

\begin{proof}
Let $u$ be a set of connector walks for $X$ and let ${}_ye_x$ be a non-zero morphism of ${}_y\C_x$.  We define its $Z_{X,u}$-degree as the homotopy class of the loop described by the walk $ {}_yu_{c_0}^{-1}, {}_ye_x,  {}_xu_{c_0}$ in $CW(\C)$, that is,
$$ \deg_{Z_{X,u}}  {}_ye_x = [  {}_yu_{c_0}^{-1} \  {}_ye_x \  {}_xu_{c_0}].$$
In order to prove that this defines a grading, let $x,y,z$ be objects in $\C$.  In case $  {}_ze_y \  {}_ye_x = 0$ there is nothing to prove.  In case
$  {}_ze_y \  {}_ye_x \not= 0$ we have that $$  {}_ze_y \  {}_ye_x =   {{}_z\lambda_x}\!\!\!^y \  {}_ze_x $$
with $ {{}_z\lambda_x}\!\!\!^y $  a non-zero element in $k$.  We have to show that the following equality holds:
$$(\deg_{Z_{X,u}}  {}_ze_y) (\deg_{Z_{X,u}}  {}_ye_x) =  \deg_{Z_{X,u}}   {}_ze_x.$$
The left hand side is the following homotopy class
%in $CW(\C)$
\begin{eqnarray*}
[ {}_zu_{c_0}^{-1} \  {}_ze_y \  {}_yu_{c_0}] [ {}_yu_{c_0}^{-1} \  {}_ye_x \  {}_xu_{c_0}] & = & [ {}_zu_{c_0}^{-1} \  {}_ze_y \  {}_yu_{c_0}  {}_yu_{c_0}^{-1} \  {}_ye_x \  {}_xu_{c_0} ] \\
& = & [ {}_zu_{c_0}^{-1} \  {}_ze_y   \  {}_ye_x \  {}_xu_{c_0} ].
\end{eqnarray*}
Observe that since  $  {}_ze_y  {}_ye_x $ is a non-zero morphism in $\C$, the CW-complex has a $2$-cell attached, which means that the path described by the walk ${}_ze_y,  {}_ye_x$ is homotopic to  $  {}_ze_x$.  This observation provides the required result. \qed
\end{proof}

\begin{lem}
Let $\C$ be a connected Schurian category with a given base object $c_0$, let $X$ be a connected grading of $\C$ and let $Z_{X,u}$ be the grading considered above by the group $\pi_1(CW(\C),c_0)$.  Let $w$ be a closed walk at $c_0$ in $\C$.  Then
$$\deg_{Z_{X,u}} w = [w] \in \pi_1 (CW(\C), c_0)$$
where $[w]$ is the homotopy class of the loop described by $w$ in $CW(\C)$.
\end{lem}

\begin{proof}
Observe first that the degree of a pure virtual morphism $( {}_ye_x, -1)$ is the homotopy class
$[ {}_yu_{c_0}^{-1} \  {}_ye_x \  {}_xu_{c_0}]^{-1}= [ {}_xu_{c_0}^{-1} \  {}_ye_x^{-1} \  {}_yu_{c_0}] $.  Hence the connector walks $ {}_xu_{c_0}$ annihilate succesively in $\pi_1(CW(\C), c_0)$, enabling us to obtain the result (recall that $ {}_{c_0}u_{c_0}=  {}_{c_0}\id_{c_0}$). \qed
\end{proof}

\begin{pro}
Let $\C$ be a connected Schurian $k$-category and let $X$ be a connected grading. Then the grading $Z_{X,u}$  obtained in Proposition \ref{zeta} is connected.
\end{pro}

\begin{proof}
Since $\C$ is connected, it is enough to prove that for any element $[l] \in \pi_1(CW(\C), c_0)$ there exists a closed walk $w$ at $c_0$ in $\C$ such that $\deg_{Z_{X,u}} w = [l].$
Recall that $[l]$ is a homotopy class, more precisely $l$ is a continuous map
$$[0,1] \to CW(\C)$$
such that $l(0)=l(1)=c_0$.  We use cellular approximation (see for instance \cite [Theorem 4.8]{H}) in order to obtain a homotopic cellular loop $l'$ such that the image of $l'$ is contained in the $1$-skeleton.  Its image is compact.  A compact set in a CW-complex meets only finitely many cells (see for instance \cite [Proposition A.1, page 520]{H}).  We infer that $l$ is homotopic to a loop $l'$ such that its image is a closed walk $w$ at $c_0$ in $\C$.  The previous Lemma asserts that the $Z_{X,u}$-degree of $w$ is precisely $[l']=[l]$. \qed
\end{proof}

\begin{defi}
Let $X$ and $Z$ be gradings of a $k$-category $\C$.  We say that $X$ is a \textbf{quotient} of $Z$ if there exists a surjective group map
$$\varphi: \Gamma\!_Z \to \Gamma\!_X$$
such that for any pair of objects $(x,y)$ we have that
$$X^s {}_y\C_x = \bigoplus_{\varphi(r)=s} Z^r  {}_y\C_x.$$
\end{defi}

\begin{thm}
Let $\C$ be a connected Schurian $k$-category and let $X$ be a connected grading of $\C$.  Let $Z_{X,u}$ be the connected grading of $\C$ by $\pi_1(CW(\C)), c_0)$ defined in the Proof of Proposition \ref{zeta}.  Then $X$ is a quotient of $Z_{X,u}$ through a unique group map $\varphi$.
\end{thm}

\begin{proof}
Let $[l]$ be a homotopy class in  $\pi_1(CW(\C)), c_0)$.  As in the previous proof, using cellular approximation we can assume that the image of $l$ is a closed walk $w$ at $c_0$ in $\C$.  In order to define a group morphism
$$\varphi:  \pi_1(CW(\C)), c_0) \to \Gamma\!_X$$
we put $\varphi([l])= \deg_Xw$. \\
 In order to check that $\varphi$ is well defined, we have to prove that $\deg_X w = \deg_X w'$ whenever $w$ and $w'$ are closed walks at $c_0$ providing homotopic loops in $CW(\C)$.
Assume first that $w$ and $w'$ only differ by a $2$-cell, that is, $ {}_ze_y , {}_ye_x $ is part of $w$, $ {}_ze_y \  {}_ye_x \not= 0$ and $w'$ coincide with $w$ except that $ {}_ze_y , {}_ye_x $ is replaced by $ {}_ze_x $ through the corresponding $2$-cell in $CW(\C)$.  Since $\C$ is Schurian we have that $ {}_ze_y \  {}_ye_x $ is a non-zero multiple of $ {}_ze_x $.  Now since $X$ is a grading
$$\deg_X ( {}_ze_y \ {}_ye_x ) = \deg_X  {}_ze_x $$
and $\deg_X w = \deg_X w'$.  \\
For the general case, let $h$ be a homotopy from $w$ to $w'$.  Using again the result in  \cite [ Proposition A.1, page 520]{H}, we can assume that the compact image of $h$ meets a finite number of $2$-cells.  Consequently $w$ and $w'$ only differ by a finite number of $2$-cells.  By induction we obtain $\deg_X w = \deg_X w'$.  \\
The map is clearly a group morphism.  In order to prove that $\varphi$ is surjective, let $s \in \Gamma\!_X$.  Since $X$ is connected, there exists a closed homogeneous walk $w$ at $c_0$ of $X$-degree $s$. Clearly there is a loop $l$ with image $w$,  hence $\varphi ([l]) = s$. \\
It remains to prove that the homogeneous component of a given $X$-degree $s$ is the direct sum of the corresponding $Z_{X,u}$-homogeneous components. Observe that since $\C$ is Schurian, the direct sum decomposition is reduced to only one component.  Let ${}_ye_x $ be a morphism which has $X$-degree $s$.  By definition, its $Z_{X,u}$-degree is $ [  {}_yu_{c_0}^{-1} \  {}_ye_x \  {}_xu_{c_0}]$ and we have to prove that $\varphi  [  {}_yu_{c_0}^{-1} \  {}_ye_x \  {}_xu_{c_0}] = s$, that is,
$\deg_X (   {}_yu_{c_0}^{-1} \  {}_ye_x \  {}_xu_{c_0}) =s$.  The result follows since the connectors $ {}_xu_{c_0}$ have trivial $X$-degree. \\
Concerning uniqueness, let $\varphi' :   \pi_1(CW(\C)), c_0) \to \Gamma\!_X$ be a surjective group map such that for each morphism ${}_ye_x $ we have $\varphi'(\deg_{Z_{X,u}} {}_ye_x) = \deg_X {}_ye_x $, that is,
$$ \varphi' \left( \left[  {}_yu_{c_0}^{-1} \  {}_ye_x \  {}_xu_{c_0}\right]\right) = \varphi  \left(\left[  {}_yu_{c_0}^{-1} \  {}_ye_x \  {}_xu_{c_0}\right]\right).$$
This shows that $\varphi$ and $\varphi'$ coincide on loops of this form.  Let now $l$ be an arbitrary loop.  In order to prove that $\varphi'([l]) = \varphi ([l])$, we first replace $l$ by a cellular approximation in such a way that $l$ describes a walk in $\C$.  Clearly any loop at $c_0$ in $CW(\C)$ is homotopic to a product of loops as above and their inverses.  We infer that $\varphi$ and $\varphi'$ are equal on any loop.  \qed
\end{proof}

We will prove next that $Z_{X,u}$  depends neither on the choice of the set $u$  nor on the connected grading $X$.
We will consider a slightly more general situation in order to prove these facts.

First recall that a set of connector walks depends on a given grading. In case there is no grading, a set of connector walks means a set of
connector walks for the trivial grading by the trivial group. In other words a set of connector walks for a linear category without a given
grading is just a choice of a set of walks from a given object $c_0$ to each object $x$, where the walk from $c_0$ to itself is  ${}_{c_0}\id_{c_0}$.

Let $\C$ be a connected Schurian $k$-category with a base object $c_0$ and let $u$ be a set of connector walks.  By definition the grading $Z_u$
of $\C$ with group $\pi_1(CW(\C),c_0)$ is given by $ \deg_{Z_u}  {}_ye_x =  [  {}_yu_{c_0}^{-1} \  {}_ye_x \  {}_xu_{c_0}].$
Next we will prove that given two sets of connector walks  $u,v$, the corresponding gradings $Z_u$ and $Z_v$ differ in a simple way that we will
call conjugation.

\begin{defi}
Let $X$ be a grading of a connected $k$-category $\C$.  Let $a=(a_x)_{x \in \C_0}$ be a set of group elements of $\Gamma\!_X$.  The \textbf{conjugated grading} ${{}{^a}\!X}$ has the same homogeneous components than $X$ but the degree is changed as follows:

\[ \left({}{^a}\!X\right)^s{}_y\C_x= X^{a_{_y} s a_{x}^{-1}}{}_y\C_x \]

\end{defi}

%\normalsize
In order to consider morphisms between gradings, they must be understood in the setting of Galois coverings, see \cite{CRS2}.  More precisely any grading gives rise to a Galois covering through a smash product construction, see \cite {CM}. The Galois coverings obtained by smash products form a full subcategory of the category of Galois coverings. Moreover, both categories are equivalent.  Consequently morphisms between gradings are morphisms between the corresponding smash product Galois coverings.

Now, to each grading $X$ of a $k$-category $\C$ we associate a new $k$-category $\C \# X$ and a functor $F_X: \C \# X \to \C$ as follows.

\begin{eqnarray*}
(\C \# X)_0 &=& \C_0 \times \Gamma\!_X \\
{}_{(y,t)} ( \C \# X )_{(x,s)} &=& X^{t^{-1} s} {}_y \C_x \\
F_X(x,s) &=& x \\
F_X &:& {}_{(y,t)} ( \C \# X )_{(x,s)}  \hookrightarrow {}_y \C_x
\end{eqnarray*}
In particular $F_X$ is a Galois covering and any Galois covering is isomorphic to one of this type.  Note that $\C \# X$ is a connected category if and only if the grading $X$ is connected.

\begin{pro}
Let $\C$ be a connected $k$-category and $X$ be a connected grading of $\C$.  Let  $a=(a_x)_{x \in \C_0}$ be a set of group elements of $\Gamma\!_X$ and ${}^a\!X$ be the conjugated grading.  The Galois coverings $\C \# X$ and $\C \# {}^a\!X$ are isomorphic, more precisely there exists a functor
$H : \C \# {}^a\!X \to \C \# X$ such that $F_X H = F_ { {}^{^a}\!\!X}$.
\end{pro}

\begin{proof}
Recall that $({}^a\!X)^s {}_y \C_x = X^{a_y s a_x^{-1}} {}_y \C_x$.  Consequently
$${}_{(y,t)} ( \C \# {}^a X )_{(x,s)} = ({}^a\!X)^{t^{-1}s} {}_y \C_x =  X^{a_y t^{-1} s a_x^{-1}} {}_y \C_x = {}_{(y,ta_y^{-1})} ( \C \# {}^a X )_{(x,sa_x^{-1})}. $$
This computation shows that defining $H$ on objects by $H(x,s)= (x, s a_x^{-1})$ and by the identity on morphisms provides the required isomorphism. \qed
\end{proof}

\begin{pro}
Let $\C$ be a connected Schurian $k$-category, $c_0$ a base object and $X,Y$ two connected gradings of $\C$.  Let $Z_{X,u}$ and $Z_{Y,v}$ be the connected gradings by the group $\pi_1(CW(\C), c_0)$, associated to the sets $u$ and $v$ of homogeneous connector walks for $X$ and $Y$ respectively, given by the choices $ {}_xu_{c_0}$ and $ {}_xv_{c_0}$ for any $x \in \C_0$.  Then $Z_{Y,v}$ and $Z_{X,u}$ are conjugated through the set of group elements $a_x =  {}_xu_{c_0}^{-1}  {}_xv_{c_0}$.
\end{pro}

\begin{proof}
Recall that
$\deg_{Z_{X,u}}  {}_ye_x = [  {}_yu_{c_0}^{-1}\  {}_ye_x \ {}_xu_{c_0}]$, then by definition
\begin{eqnarray*}
\deg_{{}^{^a}\!\!Z_{X,u}}  {}_ye_x & = & a_y^{-1} (\deg_{Z_{X,u}}  {}_ye_x ) \ a_x \\
& = & [{}_yv_{c_0}^{-1} \ {}_yu_{c_0} \ {}_yu_{c_0}^{-1} \  {}_ye_x  \ {}_xu_{c_0} \  {}_xu_{c_0}^{-1} \ {}_xv_{c_0} ] \\
&= &\deg_{Z_{Y,v}}  {}_ye_x.
\end{eqnarray*} \qed
\end{proof}

\begin{rem}
Since all the gradings $Z_{X,u}$ are isomorphic, we can choose the trivial grading by the trivial group.  However we still need to choose connector walks. Moreover we have shown that each connected grading is a unique quotient of the grading by the group $\pi_1(CW(\C), c_0)$.
\end{rem}

\begin{cor}
Let $\C$ be a connected Schurian $k$-category, and let $c_0$ be a base object.  Then
$$\pi_1(\C, c_0) = \pi_1(CW(\C),c_0).$$
\end{cor}

\begin{proof}
From \cite {CRS}, we know that in case a universal covering exists, the fundamental group of the category is its group of automorphisms.  The results we have proven show that the grading by the fundamental group of $CW(\C)$ is a universal grading, consequently the smash product Galois covering is a universal Galois covering with automorphism group $ \pi_1(CW(\C),c_0)$.\qed
\end{proof}

Next we compute the intrinsic fundamental group of a $k$-category with an infinite number of objects and without admissible presentation.

\begin{exa}

Let $\C$ be the $k$-category given by the quiver:

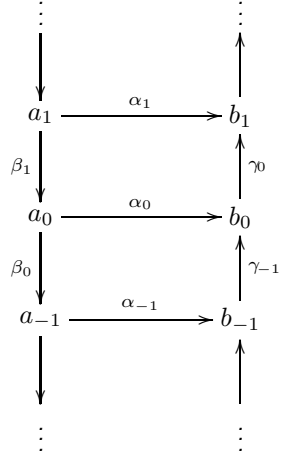
\[ \xymatrix{
\vdots   \ar[d] &  & \vdots  \\
a_1 \ar[rr]^{\alpha_1} \ar[d]_{\beta_1} & & b_1 \ar[u] \\
a_0 \ar[rr]^{\alpha_0} \ar[d]_{\beta_0} & & b_0 \ar[u]_{\gamma\!_0} \\
a_{-1} \ar[rr]^{\alpha_{-1}} \ar[d] & & b_{-1} \ar[u]_{\gamma\!_{-1}} \\
\vdots & & \vdots \ar[u]
} \]
with the relations $\gamma\!_i \alpha_i \beta_{i+1} = \alpha_{i+1}$ for all $i \not=0$ and $\gamma\!_0 \alpha_0 \beta_{1} = 0$. \\
In  $CW(\C)$ there is a $2$-cell attached to each square except the $0$-one. Consequently $\pi_1(\C)=\mathbb Z$.

\end{exa}

\section{\sf Hurewicz isomorphism}

Let $\C$ be a $k$-category.  A\textbf{ $k$-derivation} $d$ with coefficients in $\C$ is a set of linear morphisms ${}_yd_x :  {}_y \C_x \to  {}_y \C_x$ for each pair $(x,y)$ of objects, verifying
$${}_zd_x (gf) = {}_zd_y (g) f+g {}_yd_x (f) $$
for any $f \in  {}_y \C_x$ and $g \in  {}_z\C_y$. \\
Let $a=(a_x)_{x\in \C_0}$ be a family of endomorphisms of each object $x \in \C_0$, namely $a_x \in {}_x\C_x$.  The \textbf{inner derivation} $d_a$ associated to $a$ is defined by
$${}_y(d_a)_x (f) = a_y f - f a_x.$$
The \textbf{first Hochschild-Mitchell cohomology} $\HH^1(\C)$ is the quotient of the vector space of derivations by the subspace of inner ones (see \cite{Mi} for the general definition).

\begin{rem}
In fact $\HH^1(\C)$ has a Lie algebra structure, where the bracket of derivations is given by
$${}_y [d,d']_x = {}_yd_x \ {}_yd'_x - {}_yd'_x \ {}_yd_x.$$
\end{rem}

\begin{defi}
Let $X$ be a grading of a $k$-category $\C$.  The \textbf{Hurewicz morphism}
$$h: \Hom ( \Gamma\!_X, k^+) \to \HH^1(\C)$$
is defined as follows.  Let $\chi : \Gamma\!_X \to k^+$ be an abelian character and let $f$ be a homogeneous morphism in ${}_y\C_x$.  Then
$${}_y h(\chi)_x (f)= \chi (\deg_X f) f.$$
An arbitrary morphism is decomposed as a sum of its homogeneous components in order to extend linearly the definition of ${}_y h(\chi)_x$.
\end{defi}

\begin{rem}
The set $h(\chi)$ is a derivation.  This can be verified in a simple way relying on the fact that $X$ is a grading.  Derivations of this type are called "Eulerian derivations", see for instance \cite {FGM, FGGM}.
\end{rem}

The following result is immediate.

\begin{lem}
The image of the Hurewicz morphism is an abelian Lie subalgebra of $\HH^1(\C)$.
\end{lem}

We recall that, under some assumptions, the Hurewicz morphism is injective.

\begin{pro}
Let $\C$ be a $k$-category and assume the endomorphism algebra ${}_x \C_x$ of each object $x$ in $\C_0$ is equal to $k$.  Let $X$ be a connected grading of $\C$.  Then the Hurewicz morphism is injective.
\end{pro}

\begin{proof}
If $h(\chi)$ is an inner derivation,
$${}_{t(f)} h(\chi)_{s(f)} (f) = \chi (\deg_X f) f = a_{t(f)} f - f a_{s(f)} $$
for any homogeneous non-zero morphism $f$, where $(a_x)_{x \in \C_0}$ is a set of endomorphisms which are elements of $k$ by hypothesis.  Then $\chi(\deg_X f) = a_{t(f)} - a_{s(f)}$. \\
Now we assert that the same equality holds for any homogeneous walk $w$, that is,
$$\chi(\deg_X w) = a_{t(w)} - a_{s(w)}.$$
For instance let $w= (g,-1), (f,1)$ be a homogeneous walk where $f \in {}_y \C_x$ and $g \in {}_y\C_z$.  Then

\begin{eqnarray*}
\chi (\deg_X w) & = &  \chi ( (\deg_X g )^{-1}(\deg_Xf)) = - \chi (\deg_X g) + \chi (\deg_X f) \\
              & = & a_{s(g)} - a_{t(g)} + a_{t(f)} - a_{s(f)} \\
              & = &  a_z - a_y + a_y - a_x = a_z - a_x = a_{t(w)} - a_{s(w)}.
\end{eqnarray*}
Let $c_0$ be any fixed object of $\C$.  Since $X$ is a connected grading, for any group element $s \in \Gamma\!_X$ there exists a homogeneous walk $w$, closed at $c_0$, such that $\deg_X w = s$.  Consequently
$$\chi(s) w = (a_{c_0} - a_{c_0}) w =0$$
hence $\chi(s) = 0$ for any $s \in \Gamma\!_X$. \qed
\end{proof}

\begin{thm}
Let $\C$ be a connected Schurian $k$-category and let $U$ be its universal grading by the fundamental group $\pi_1(CW(\C),c_0)$.  The corresponding Hurewicz morphism is an isomorphism.
\end{thm}

\begin{proof}
The previous result insures that $h$ is injective.  In order to prove that $h$ is surjective, let $d$ be a derivation.  We choose a non-zero morphism ${}_ye_x$ in each $1$-dimensional space of morphisms ${}_y\C_x$, with ${}_xe_x = {}_x\id_x$. Let $c_0$ be a fixed object in $\C$. To describe the universal grading, recall that we choose a set of  connector walks, hence
$$\deg_U {}_ye_x = [ {}_yu_{c_0}^{-1} \  {}_ye_x  \ {}_xu_{c_0} ] \in \pi_1(CW(\C),c_0).$$
Since ${}_y\C_x$ is one dimensional, $d( {}_ye_x ) = {}_y\lambda_x \ {}_ye_x$ with ${}_y\lambda_x \in k$.
In order to define an abelian character $\chi$ such that $h(\chi)=d$, let $l$ be a loop at $c_0$ in $CW(\C)$.  By cellular approximation we can assume that the image of $l$ is a closed walk $w$ in $\C$.
In case $w$ is of the form $ {}_yu_{c_0}^{-1} \  {}_ye_x \ {}_xu_{c_0} $ we define $\chi [l] = {}_y \lambda_x$.  Otherwise the cellular loop $w$ is homotopic to a product of loops of the previous type or of their inverses, and we define $\chi [l] $ to be the corresponding sum of scalars.
We have to verify that $\chi$ is well defined.  First observe that if ${}_ze_y \ {}_ye_x \not = 0$, the scalars of the derivation $d$ verify
$$  {}_z\lambda_x =  {}_z\lambda_y +  {}_y\lambda_x.$$
Indeed, ${}_ze_y \ {}_ye_x  = \mu \ {}_ze_x$, with $\mu \not = 0$, hence

\begin{eqnarray*}
d ( {}_ze_y \ {}_ye_x )  & = &   {}_ze_y \  d( {}_ye_x) + d({}_ze_y) \  {}_ye_x \\
              & = & \mu \ (   {}_z\lambda_y +  {}_y\lambda_x  )  \ {}_ze_x .
\end{eqnarray*}
We deduce the result since $d( \mu \ {}_ze_x) = \mu \ {}_z\lambda_x \ {}_ze_x$.
Consider now two cellular loops $l$ and $l'$ which are homotopic by a $2$-cell, meaning that a walk ${}_ze_y , {}_ye_x$ is replaced by ${}_ze_x$.  The previous computation shows that $\chi [l ] = \chi [l']$.  We have already verified that any homotopy of cellular loops decomposes as a finite number of homotopies of the previous type, hence we deduce that $\chi$ is a well defined map.  By construction $\chi : \pi_1(CW(\C), c_0) \to k^+$ is an abelian character and clearly $h(\chi)=d$. \qed
\end{proof}

%%%%%%%%%%%%%%%%%%%%%%%%%%%%%%%%%%%%%%%%%%%%%%%%%%%%%%%%%%%%%%%%%%%%%%%%%%%%%%%%%%%%%%%%%%%

\footnotesize
\noindent C.C.:
\\Institut de math\'{e}matiques et de mod\'{e}lisation de Montpellier I3M,\\
UMR 5149\\
Universit\'{e}  Montpellier 2,
\\F-34095 Montpellier cedex 5,
France.\\
{\tt Claude.Cibils@math.univ-montp2.fr}

\noindent M.J.R.:
\\Departamento de Matem\'atica,
Universidad Nacional del Sur,\\Av. Alem 1253\\8000 Bah\'\i a Blanca,
Argentina.\\ {\tt mredondo@criba.edu.ar}

\noindent A.S.:
\\Departamento de Matem\'atica,
 Facultad de Ciencias Exactas y Naturales,\\
 Universidad de Buenos Aires,
\\Ciudad Universitaria, Pabell\'on 1\\
1428, Buenos Aires, Argentina. \\{\tt asolotar@dm.uba.ar}


\begin{thebibliography}{99}


\bibitem{asde} Assem, I.; de la Pe\~{n}a, J. A. The fundamental groups of a
triangular algebra.  Comm. Algebra  \textbf{ 24}  (1996),  187--208.

\bibitem{baza} Bahturin, Y.; Zaicev, M.
Group gradings on matrix algebras,
Canad. Math. Bull. \textbf{ 45} (2002), no. 4, 499--508.

\bibitem{ba}
Bardzell, M. J.; Marcos, E. N.
$\H^1$ and presentations of finite dimensional algebras. Representations of algebras (S\~ao Paulo, 1999), 31--38,
Lecture Notes in Pure and Appl. Math., \textbf{ 224}, Dekker, New York, 2002.

\bibitem{bodas07} Boboc, C. ; D\u asc\u alescu, S.
Group gradings on $M_3(k)$,
Comm. Algebra  \textbf{ 35} (2007), 2654\"{\i}?`\"{\i}?`½-2670.

\bibitem{bo}
Bongartz, K.; A criterion for finite representation type.
Math. Ann. \textbf{ 269} (1984), no. 1, 1--12.

\bibitem{boga} Bongartz, K.; Gabriel, P. Covering spaces in representation-theory, Invent. Math. \textbf{ 65} (1981/82) 331--378.

\bibitem{bg} Bretscher, O.; Gabriel, P. The standard form of a representation-finite algebra. Bull. Soc. Math. France  \textbf{ 111} (1983),  no. 1, 21--40.

\bibitem {bu1} Bustamante, J.C. On the fundamental group of a Schurian algebra.  Comm. Algebra  \textbf{ 30}  (2002),  no. 11, 5307--5329.

\bibitem{bu2} Bustamante, J.C.  The classifying space of a bound quiver.  J. Algebra  \textbf{ 277}  (2004),  no. 2, 431--455.

\bibitem{buca}Bustamante, J.C.; Castonguay, D. Fundamental groups and presentations of algebras. J. Algebra Appl. \textbf{ 5} (2006), 549--562.

\bibitem{CM} Cibils, C.; Marcos, E. Skew category, Galois
covering and smash product of a category over a ring. Proc. Amer.
Math. Soc.  \textbf{ 134}, (2006),  no. 1, 39--50.

\bibitem{cire} Cibils, C.; Redondo M.J. Cartan-Leray spectral sequence for Galois coverings
of categories. J. Algebra \textbf{ 284} (2005), 310--325.

\bibitem{CRS} Cibils, C.; Redondo M.J.; Solotar, A. The intrinsic fundamental group of a linear category.
To appear in Algebras and Representation Theory. DOI: 10.1007/s10468-010-9263-1.

\bibitem{CRS2} Cibils, C.; Redondo M.J.; Solotar, A. Connected gradings and fundamental group. Algebra Number Theory \textbf{ 4} (2010), no. 5, 625--648.

\bibitem{ciso} Cibils,C.; Solotar, A. Galois coverings, Morita equivalence and smash extensions of categories over a
field. Doc. Math. \textbf{ 11} (2006), 143--159.

\bibitem{dodo}Douady, R.; Douady, A. Alg\`ebre et th\'{e}ories galoisiennes.  Paris: Cassini. 448 p.
(2005).

\bibitem{doha}Dowbor, P; Hajduk,A. On some nice class of non-Galois coverings for algebras, Bol. Soc. Mat. Mex. \textbf{39} (in press).

\bibitem{FGM} Farkas, D.; Green, E.L.; Marcos, E.N. Diagonalizable derivations of finite-dimensional algebras. II. Pacific J. Math. \textbf{ 196} (2000), no. 2, 341--351.

\bibitem{FGGM} Farkas, D.R.; Geiss, Ch.; Green, E.L.; Marcos, E.N. Diagonalizable derivations of finite-dimensional algebras. I. Israel J. Math. \textbf{ 117} (2000), 157--181.

\bibitem{ga}
Gabriel, P. The universal cover of a representation-finite algebra.
Representations of algebras (Puebla, 1980), 68--105, Lecture
Notes in Math. \textbf{ 903}, Springer, Berlin-New York, 1981.

\bibitem{gr}
Green, Edward L. Graphs with relations, coverings and group-graded algebras.
Trans. Amer. Math. Soc. \textbf{ 279} (1983), no. 1, 297--310.

\bibitem{H} Hatcher, A. Algebraic topology. Cambridge University Press, Cambridge, 2002. xii+544 pp.

\bibitem{kabodas} Khazal, R.;  Boboc, C.;  D\u asc\u alescu, S.
Group gradings of $M\sb 2(K)$,
Bull. Austral. Math. Soc. \textbf{ 68} (2003), no. 2, 285--293.

\bibitem{le}Le Meur, P. The universal cover of an algebra without double bypass,
J. Algebra  \textbf{ 312}  (2007),  no. 1, 330--353.

\bibitem{le1}Le Meur, P.
The fundamental group of a triangular algebra without double bypasses. C. R. Math. Acad.
Sci. Paris 341 (2005), 211--216.

\bibitem{le2}Le Meur, P. Rev\^{e}tements galoisiens et groupe fondamental d'alg\`ebres de dimension
finie. Ph.D. thesis, Universit\'{e} Montpellier 2 (2006).
\texttt{http://tel.archives-ouvertes.fr/tel-00011753}

\bibitem{MP}
Mart\'\i nez-Villa, R.;  de la Pe\~na, J. A. The universal cover of a quiver with
relations. J. Pure Appl. Algebra \textbf{ 30} (1983), 277--292.

\bibitem{madelp}
Martins, Ma. I. R.; de la Pe\~na, J. A.
Comparing the simplicial and the Hochschild cohomologies of a finite-dimensional algebra. J. Pure Appl. Algebra \textbf {138} (1999), no. 1, 45--58.

\bibitem{Mi} Mitchell, B. Rings with several objects. Adv. Math. \textbf {8} (1972), 1--161.

\bibitem{japre} de la Pe\~na, J.A.; Redondo, M.J. Coverings without groups. J. Algebra \textbf{321} (2009), 3816--3826.

\bibitem{dlpsa} de la Pe\~na, J.A.; Saor\'\i n, M. On the first Hochschild cohomology group of an
algebra. Manuscripta Math. \textbf{104} (2001), 431--442.

\bibitem{Re} Reynaud, E. Algebraic fundamental group and simplicial complexes.  J. Pure Appl. Algebra  \textbf{ 177}  (2003),  no. 2, 203--214.

\end{thebibliography}
\end{document}